\documentclass[12pt]{article}  

\usepackage{graphicx}    
\usepackage{amsmath}     
\usepackage{amscd} 
\usepackage{amssymb}     
\usepackage{amsfonts}     
\usepackage{amsthm}  
\usepackage{makeidx}      
\usepackage{verbatim}    
\usepackage{epsfig}      
\usepackage{tabls}  
 
\usepackage{hyperref} 
 
\newcommand{\N} {{\mathbb N}}

\newcommand{\e}{\varepsilon }       
               
\renewcommand{\a}{\alpha }

\newcommand{\reals} {{\mathbb R}}      
       
\newcommand{\nat} {{\mathbb N}}

\renewcommand\inf{\operatornamewithlimits{inf\vphantom{p}}}

\renewcommand{\epsilon}{\varepsilon}     
\renewcommand{\rho}{\varrho}     
\renewcommand{\phi}{\varphi}     
\renewcommand{\theta}{\vartheta}     
\newcommand{\wt}{\widetilde }

\renewcommand{\subset}{\subseteq}

\renewcommand\ker{{\rm ker}}

\newcommand\lall{\Lambda^{{\rm all}}}    
\newcommand\lstd{\Lambda^{{\rm std}}}

\newcommand\NM{{\cal N}}

\renewcommand\inf{\operatornamewithlimits{inf\vphantom{p}}}

\def\tlall{\widetilde \Lambda^{\rm all}}     

\newcommand{\enwall}{e^{\rm \, all-wor}_n}  
  
\newcommand{\enrall}{e^{\rm \, all-ran}_n}  
  
\newcommand{\enwALL}{\wt e^{\rm \, all-wor}_n}

\theoremstyle{definition}  
 
\newtheorem{thm}{Theorem} 
\newtheorem{rem}{Remark} 
\newtheorem{lem}{Lemma} 
 
\newtheorem{cor}{Corollary}

\newtheorem{defi}{Definition}  

\begin{document}  

%
%
%


\title{Discontinuous information
in the worst case and randomized settings}  

\author{Aicke Hinrichs, Erich Novak, Henryk Wo\'zniakowski} 

\date{May 12, 2011} 

\maketitle 

\begin{center}
{Dedicated to Hans Triebel on the occasion of his 75th birthday} 
\end{center} 

\begin{abstract}     
We believe that discontinuous linear information is never  
more powerful than continuous linear information 
for approximating  continuous operators.  
We prove such a result in the worst case setting. 
In the randomized setting we consider 
compact linear operators defined between Hilbert spaces.  
In this case, the use of discontinuous linear information 
in the randomized setting  
cannot be much more powerful than continuous linear information 
in the worst case setting.   
These results can be applied when function evaluations are used 
even if function values are defined only almost everywhere. 
\end{abstract}      

\maketitle  

\section{Introduction}  \label{s1}  
   
We study the approximation of an operator $S$ defined between  
normed spaces $F$ and $G$.  
The operator $S$ does not have to be linear or continuous.  
We approximate $S(f)$ by algorithms that use  
information consisting of finitely many 
continuous or discontinuous linear functionals $L_i : F \to \reals$.  
The error of such algorithms is defined either in the worst case or 
randomized setting. 
 
For continuous $S$, it is hard  
to imagine that one can learn about $S(f)$  
by using discontinuous information. 
On the other hand, it is well known that the Monte Carlo  
algorithm works nicely for multivariate integration defined 
for $L_2$-functions. This algorithm uses linear functionals given by   
function evaluations which are indeed discontinuous  
or even not always well defined. 
Hence, discontinuous information  
is actually used in computational practice  
and seems to be useful, at least in the randomized setting.  
 
This is the subject of this paper. We want to verify the power  
of \emph{discontinuous} linear information and compare it to the power  
of \emph{continuous} linear approximation. We study the worst case and 
randomized settings. This is done by comparing  
the $n$th minimal (worst case and randomized) errors  
which we can achieve  by using $n$ discontinuous or continuous linear 
functionals.  
 
In the worst case setting, we prove that as long as $S$ is a 
continuous operator (not necessarily linear) then the $n$th minimal 
errors are exactly the same for the class $\tlall$ of all 
discontinuous or continuous  
linear functionals and the class $\lall$ of all continuous linear 
functionals, see Theorem~\ref{adaptive}. This means that the use of 
discontinuous linear functionals does not help. The situation is quite 
different if $S$ is discontinuous. We present an easy example  
of a discontinuous linear functional $S$ for which  the $n$th minimal 
errors for the class~$\tlall$ are zero for all $n\ge1$, whereas  
the $n$th minimal errors for the class $\lall$ are infinity  
for all $n\ge1$.   
 
In the randomized setting, we mostly consider compact linear operators 
$S$ defined between Hilbert spaces~$F$ and $G$. In this case, we know 
from~\cite{No92}  
that the power of continuous linear functionals in the randomized 
setting is roughly the same as in the worst case setting, see 
Lemma~\ref{lemma7}. 
Here, the word ``roughly'' means that the $n$th minimal error in the 
randomized setting is at least as large as a half of the $(4n-1)$st 
minimal error in the worst case setting, and obviously it is at most 
as large as the $n$th minimal error in the worst case setting.   
By combining with
the result from the worst case setting, we conclude that the power of 
discontinuous linear functionals in the randomized setting is 
roughly the same as the power of continuous linear functionals in the 
worst case setting. On the other hand, if we drop the assumption  
that $S$ is a compact linear operator between Hilbert spaces then 
we can construct a problem~$S$ which is \emph{not} solvable in the worst 
case setting  and solvable and relatively easy in the randomized 
setting. Here, not solvable means that the $n$th minimal errors  
in the worst case setting  
do \emph{not} converge to zero, and relatively easy means the $n$th 
minimal errors in the randomized setting are of order $n^{-1/2}$. 
 
For many applications the class $F$ consists of functions and 
we can only use function evaluations 
for the approximation of $S$. The class of such evaluations is called 
\emph{standard} and denoted by $\lstd$.   
These evaluations are always linear but not always continuous. 
That is, we always have $\lstd\subseteq \tlall$, and, 
depending on the space $F$, we sometimes have $\lstd\subseteq \lall$.  
In either case, our results apply. 
In particular, if all function 
evaluations are discontinuous then they may be 
useless in the 
worst case setting since the minimal 
worst case error of any algorithm that uses $n$ function values is as 
good as a constant algorithm that uses \emph{no} function values, 
see Remark~\ref{fevalworst}. 

For some applications the space $F$ consists of 
equivalence classes of functions that are equal almost everywhere. 
This is the case for $F_1=L_2(D)$ for
some $D\subset\reals^d$. Then function evaluations are not even well 
defined. We extend our analysis also to such function evaluations
and show that again the same results as before hold. 
 
\section{Worst Case Setting}  \label{s2}  

For arbitrary normed spaces $F$ and $G$,  
consider an arbitrary operator $S:F\to G$ 
that does not have to be linear or continuous. 
We approximate $S(f)$ for $f$ from the unit ball of $F$  
by algorithms that use finitely many linear  
functionals from $\lall$ or from $\tlall$, respectively.  
More precisely, we consider   
algorithms $A_n:F\to G$ given by  
\begin{equation}  \label{eq01}  
A_n(f)=\phi_n(L_1(f),L_2(f),\dots,L_n(f)),  
\end{equation}  
where $n$ is a nonnegative integer, $\phi_n:\reals^n\to G$ is an  
arbitrary mapping, and $L_j\in \Lambda$, where  
$\Lambda \in \{\lall, \tlall \}$.  
Hence, for $\Lambda=\lall$ we only use \emph{continuous} linear  
functionals, whereas for $\Lambda=\tlall$ we may also use 
\emph{discontinuous} linear functionals. 
 
The choice of $L_j$ can be \emph{nonadaptive} or \emph{adaptive}. 
It is nonadaptive if the functionals $L_j$ are the same for all $f\in 
F$, and it is adaptive if $L_j$  depends on the  
already computed values $L_1(f),L_2(f),\dots,L_{j-1}(f)$.  
That is, 
$$ 
N(f)=(L_1(f),L_2(f),\dots, L_n(f)) 
$$ 
is the information used by the algorithm $A_n$ and 
$L_j=L_j(\,\cdot\, ;L_1(f),L_2(f),\dots,L_{j-1}(f))\in\Lambda$.  
If the choice of all $L_j$'s  is independent of $f\in F$ then $N$  
is \emph{nonadaptive} information, otherwise if at least one $L_j$ 
varies with $f\in F$ then $N$ is \emph{adaptive} information.    
For $n=0$, the mapping $A_n$ is a constant element of the space~$G$.   
More details can be found in e.g.,~\cite{No88,NW08,TWW88}.   
We define the error of such algorithms  
by taking the worst case setting, i.e.,  
$$ 
e(A_n) = \sup_{\Vert f \Vert_F < 1}  
\big\Vert S(f)  - A_n(f)\big \Vert_G. 
$$
Observe that it is enough that the operator $S$ is defined on the open
unit ball in $F$, not necessarily on the whole space $F$.
We take the open unit ball instead of the more standard case of the 
closed unit ball
of $F$ in the definition of the worst case error since this includes
also operators with singularities on the boundary of the unit ball.
For linear continuous $S$, or more generally 
for $S$ uniformly continuous on
the closed unit ball of $F$, this does not make a difference.  

We define the $n$th minimal errors of approximation of $S$ in the 
worst case setting as follows.   

\begin{defi}   
For $n=0$ and $n\in \nat:=\{1,2,\dots \}$, let  
$$  
\enwall(S) = \inf_{A_n \ {\rm with}\ L_j\in \lall}   
e(A_n)  
$$  
and   
$$  
\enwALL(S) = \inf_{A_n\ {\rm with}\ L_j\in \tlall}   
e(A_n) .  
$$  
\end{defi}   
 
For $n=0$, we obtain 
$$  
e^{\rm \, all-wor}_0(S)= \wt e^{\rm \,   all-wor}_0(S)=  
\inf_{g\in G}\ \sup_{\Vert f \Vert_F < 1} \Vert S(f)-g \Vert_G.   
$$   
It is easy to see that the best algorithm  
is $A_0=0$ if we assume that $S(f)=-S(-f)$ for all $\|f\|_F<1$. 
Then  
$$  
e^{\rm \, all-wor}_0(S) 
= \wt e^{\rm \,  all-wor}_0(S)=  
\sup_{\Vert f \Vert_F < 1} \Vert S(f)  \Vert_G.   
$$   
The error $e^{\rm \ all-wor}_0(S)$ 
is the initial error that can be achieved without   
computing any linear functional on the elements $f \in F$.  
Clearly,  
$$  
\enwALL(S) \le \enwall(S)\ \ \ \mbox{for all} \ \ \  n\in\nat.   
$$   
The sequences $\left\{\enwALL(S)\right\}$ and $\left\{\enwall(S)\right\}$  
are both non-increasing but not necessarily convergent to zero.   
 
We will use the following fact from  
functional analysis, see, e.g., \cite[Ch. 3]{Bol90}.   
 
\begin{lem}  \label{lemma2}  
Assume that $F$ is a normed space and $L \in \tlall$  
is discontinuous.  Then for all real $\alpha$ the set 
$\{ f \in F \mid L(x) = \alpha  \}$ \ is dense in \ $F$.  
\end{lem}  

We are ready to prove that discontinuous linear functionals do not 
help for the approximation of continuous operators in the worst case 
setting.   

\begin{thm}\label{adaptive} 
Let $F$ and $G$ be normed spaces and let $S:F\to G$ be 
continuous. Then 
$$ 
e_n^{\rm \, all-wor}(S) = \widetilde e_n^{\rm \, all-wor}(S).  
$$ 
\end{thm}  
 
\begin{proof} 
We may assume that ${\rm dim}(F)=\infty$ since otherwise 
$\tlall=\lall$ and there is nothing to prove.  
Consider arbitrary adaptive information $N=(L_1,L_2,\dots,L_n)$ 
with  
$$L_j=L_j(\cdot;y_1,y_2,\cdots,y_{j-1})\in \tlall\ \ \  
\mbox{and}\ \ \  
y_i=L_i(f;y_1,y_2,\dots,y_{i-1})\ \  \mbox{for}\ \ f\in F. 
$$ 
It is well known that the infimum of the worst case errors of 
algorithms $A_n$ that use information~$N$ is given by the radius of  
information $N$, 
\begin{equation}\label{radius} 
r(N)=\sup_{y\in N(F)}\, {\rm rad}(\{S(f)\ |\ \ N(f)=y,\,\|f\|_F<1\}), 
\end{equation} 
where ${\rm rad}(A)=\inf_{g\in G}\sup_{a\in A}\|g-a\|$ denotes the 
radius of a set $A\subset G$, see \cite{TWW88}. Without loss of 
generality we may assume that the linear functionals 
$L_1,L_2,\dots,L_n$ are linearly independent since otherwise the 
choice of a linear functional, say, $L_j$ that is linearly dependent 
on $L_1,L_2,\dots,L_{j-1}$ does not increase our knowledge about the 
element $f$. This implies that we may assume that $N(F)=\reals^n$. 
 
For $y\in N(F)=\reals^n$ and $j=1,2,\dots,n$, define  
$$ 
B_k=B_k(y)=\{f\in F\ |\ \  
L_j(f)=y_j\ \ \mbox{for}\ \ j=1,2,\dots,k\}, 
$$ 
Then the $B_k$ are affine subspaces of $F$. 
 
For each affine subspace $B$ of $F$, we associate the uniquely 
determined linear subspace $\wt B$ such that $B = f + \wt B$ for 
any $f\in B$.  
It is easily seen that a linear functional on $F$ 
is continuous on $B$ if and only if it is continuous on $\wt B$. 
In particular, 
$$ 
  \wt B_k =\wt B_k(y) = \ker L_1 \cap \ker L_2 \cap \dots \cap \ker L_k, 
$$  
and continuity of $L_{k+1}$ on $B_k$ is equivalent to continuity  
of $L_{k+1}$ on $\wt B_k$.  
 
We may now further assume for $k=1,\dots,{n-1}$ that the functional 
$L_{k+1}$ satisfies the following condition: 
\begin{center} 
either $L_{k+1}$ is continuous on $F$ \ or \  
$L_{k+1}$ is discontinuous on $B_k$. 
\end{center} 
Indeed, assume that $L_{k+1}$ is continuous on $B_k$. 
Then it is also continuous on $\wt B_k$. Let $L$ be a continuous 
linear extension of $L_{k+1}$ on $\wt B_k$ to the whole space~$F$.  
Then $L_{k+1} - L$ is a linear functional which is 0 on $\wt B_k$, so 
that $\ker (L_{k+1} - L ) \supset \wt B_k$. 
This implies that $L_{k+1} - L$ is in the span of $L_1,L_2,\dots,L_k$, 
and for some numbers $a_j$ we have  
$L(f)=L_{k+1}(f)+\sum_{j=1}^ka_jL_j(f)$ for all $f$.  
Hence knowing $L_j(f)$ for $j=1,2,\dots,k$, we know $L(f)$ iff we know 
$L_{k+1}(f)$.  
This means that we can replace the functional $L_{k+1}$  
in the information $N$ with the  
continuous functional $L$ without essentially changing  
the information and without changing its radius. 
 
We now define the information $N^*=(L_1^*,L_2^*,\dots,L_n^*)$ with 
adaptively chosen  
$$ 
L^*_j=L_j^*(\cdot;y^*_1,y^*_2,\dots,y^*_{j-1})\in\lall\ \ \  
\mbox{and}\ \ \  y^*_i=L_i^*(f;y^*_1,y^*_2,\dots,y^*_{i-1})\ \  
\mbox{for}\ \  f\in F 
$$  
such that $r(N^*)\le r(N)$. Since $N$ is arbitrary and $N^*$ consists  
of continuous linear functionals, this will prove the theorem. 
 
The functionals $L_j^*$ are defined inductively.  
We define $L^*_1=0$ if $L_1$ is discontinuous on~$F$, and otherwise  
we take $L^*_1=L_1$.  
Observe that in the case $L_1^*=0$ the next functional $L_2^*$  
cannot be chosen adaptively since $L_1^*(f)=0$ for all $f \in F$.  
Therefore, in general, $L_2^*=L_2^*(\cdot,0)$ will be different from  
$L_2=L_2(\cdot;y_1)$, with $y_1=L_1(f)$, even if $L_2$ is continuous.  
 
Assume now inductively that for all 
$j=1,2,\dots,k<n$, we have already defined 
$$ 
L_j^*=L_j^*(\cdot;y^*_1,y^*_2,\dots,y^*_{j-1})\in\lall\ \ \  
\mbox{with}\ \ \  
y^*_j=L^*_j(f;y^*_1,y^*_2,\dots,y^*_{j-1})\ \  
\mbox{for all}\ \  f \in F. 
$$  
 
Let $L_{k+1}=L_{k+1}(\cdot;y^*_1,y_2^*,\dots,y^*_k)\in \tlall$  
be the next linear functional for the original information~$N$.  
Let $B_k = B_k(y^*)$ be defined as above for $y=y^*$. 
Define also 
$$ 
A_k=A_k(y^*)=\{f\in F\ |\ \  
L^*_j(f)=y^*_j\ \ \mbox{for}\ \ j=1,2,\dots,k\} .
$$ 
If $L_{k+1}$ is continuous, we set $L^*_{k+1}=L_{k+1}$, 
if $L_{k+1}$ is discontinuous on $B_{k}$, we define  
$L^*_{k+1}=0$. In the latter case only $y^*_{k+1}=0$ 
needs to be further considered. 
As shown above, this completes the definition of  
$N^*=(L_1^*,L_2^*,\dots,L_n^*)$ that consists of 
$n$ continuous linear functionals~$L_j^*$.  
 
We now show inductively that $B_j\subseteq A_j$ and $B_j$ is dense in $A_j$ 
for all $j=1,2,\dots,n$.   
 
Indeed, for $j=1$ we have $B_1=A_1$ if $L_1$ is continuous on $F$, and  
$$ 
B_1=\{f\in F\ |\ \ L_1(f)=0\}\subseteq A_1=F 
$$ 
if $L_1$ is discontinuous on $F$.  
{}From Lemma~\ref{lemma2} we know that $B_1$ is dense 
in $A_1$.  
 
Assume now that $B_j\subseteq A_j$ and $B_j$ is dense in 
$A_j$ for $j=1,2,\dots,k<n$ with $k\ge1$.  
Consider first the case when $L_{k+1}$ is continuous.  
Then we have $L^*_{k+1}=L_{k+1}$ and 
\begin{eqnarray*} 
A_{k+1}&=&\{f\in A_k\ |\ \ L^*_{k+1}(f)=y^*_{k+1}\}\\ 
B_{k+1}&=&\{f\in B_k\ |\ \ L_{k+1}(f)=y^*_{k+1}\}=\{f\in B_k\ |\ \  
L^*_{k+1}(f)=y^*_{k+1}\}. 
\end{eqnarray*} 
Hence 
$$ 
A_{k+1}=A_k\,\cap\, C_k\ \ \ \mbox{and}\ \ \ B_{k+1}=B_k\,\cap\,C_k\ \ \ 
\mbox{with}\ \ \ 
C_k=\{f\in F\ | \ \ L^*_{k+1}(f)=y^*_{k+1}\}. 
$$ 
Clearly, $B_k\subseteq A_k$ implies that $B_{k+1}\subseteq A_{k+1}$. 
We need to show that if $B_k$ is dense in $A_k$ then $B_k\cap C_k$ is 
dense in $A_k\cap C_k$. This is obvious if $L^*_{k+1}=0$. Assume then  
that $L^*_{k+1}\not=0$.  
Take $f\in A_k\cap C_k$. Then for any positive 
$\e$ there exists $f_\e\in B_k$ such that $\|f-f_\e\|\le\e$.  
For $g_{k+1}\in F$ with $L_j(g_{k+1})=0$ for 
$j=1,2,\dots,k$ and $L^*_{k+1}(g_{k+1})=1$, define 
$$ 
\tilde{f}_\e=f_\e+\left(y^*_{k+1}-L^*_{k+1}(f_\e)\right)\,g_{k+1}. 
$$ 
Then $\tilde{f}_\e\in B_k\cap C_k$ and since $L^*_{k+1}(f)=y^*_{k+1}$ 
we have  
$$\|f-\tilde{f}_\e\|\le 
\|f-f_\e\| +|L^*_{k+1}(f)-L^*_{k+1}(f_\e)|\,\|g_{k+1}\| 
\le \|f-f_\e\|\left(1+\|L^*_{k+1}\|\,\|g_{k+1}\|\right). 
$$ 
Hence  $B_k\cap C_k$ is dense in $A_k\cap C_k$, as needed.  
 
Consider now the case when $L_{k+1}$ is discontinuous on $B_k$. Then 
$L^*_{k+1}=0$ and $y^*_{k+1}=0$. We now have  
$$ 
A_{k+1}=A_k\ \ \ \mbox{and}\ \ \ B_{k+1}=B_k\,\cap\,C_k\ \ \ 
\mbox{with}\ \ \ 
C_k=\{f\in F\ | \ \ L_{k+1}(f)=0\}. 
$$ 
Clearly $B_{k+1}\subseteq B_k\subset A_k=A_{k+1}$.  
Since $L_{k+1}$ is discontinuous on $B_k$ it is also discontinuous on 
$F$. Then Lemma~\ref{lemma2} says that $C_k$ is a linear subspace 
which is dense in $F$. We want to  
show that $B_k\cap C_k$ is dense in $A_k=A_{k+1}$. 
Similarly as before, we take $f\in A_k$. For any positive $\e$ we can 
find $f_\e\in B_k$ such that $\|f-f_\e\|\le\e$. If $L_{k+1}(f_\e)=0$  
then $f_\e\in B_k\cap C_k$ and we are done. Assume then that 
$L_{k+1}(f_\e)\not=0$. We now choose  
$$ 
 g_{k+1}\in \wt B_k  = \ker L_1 \cap \ker L_2 \dots \cap \ker L_k 
 \ \ \ \mbox{with}\ \ \ L_{k+1}(g_{k+1})=1. 
$$  
Since $L_{k+1}$ is discontinuous on $B_k$, 
it is also discontinuous on $\wt B_k$, 
so the element $g_{k+1}$  
can be of an arbitrary small norm. We choose a nonzero $g_{k+1}$ 
such that $\|g_{k+1}\|\le \e/|L_{k+1}(f_\e)|$. Then 
$$ 
\tilde{f}_\e=f_\e-L_{k+1}(f_\e)\,g_{k+1} 
$$    
belongs to $B_k\cap C_k$ and  
$$ 
\|f-\tilde{f}_\e\|\le \|f-f_\e\|+|L_{k+1}(f_\e)|\,\|g_{k+1}\|\le 2\e. 
$$ 
Hence, $B_k \cap C_k$ is dense in $A_k$, as needed. 
This completes the proof that $B_n=N^{-1}(y^*)$ is dense in 
$A_n=(N^*)^{-1}(y^*)$.  

Due to continuity of $S$, the set   
$B(y^*):=\{S(f)\in G\ |\  N(f)=y^*,\ \|f\|_F<1\}$ is dense in  
$A(y^*):=\{S(f)\in G\ |\ N^*(f)=y^*,\ \|f\|_F<1\}$  
and therefore ${\rm rad}(A(y^*))={\rm rad}(B(y^*))$. 
This holds for all $y^*\in N^*(F)$ and therefore $r(N^*)\le r(N)$, as   
needed.  
\end{proof} 
 
The assumption on continuity of $S$ in Theorem~\ref{adaptive} is 
needed. Indeed,  
assume that ${\rm dim}(F)=\infty$. Then there are 
discontinuous linear functionals $L:F\to \reals$.  
Define $S=L$. Note that now 
$e^{\rm \ all-wor}_0(S)= \wt e^{\rm \ all-wor}_0(S)= \infty$.  
 
Clearly, the worst case error $A_1(f)=L(f)=S(f)$ 
is zero. Therefore  
$\enwALL(S)=0$ for all $n\ge1$. On the other hand,  
we know that adaption does not help for linear functionals as 
proved by Bakhvalov. Furthermore,  
if we use $n$  
linear continuous nonadaptive functionals $L_j$  
then Smolyak's theorem tells us  
that the best $\phi$ in~\eqref{eq01} is linear, i.e., there are some real 
numbers $a_j$ for which the algorithm   
$A_n(f)=\sum_{j=1}^na_jL_j(f)$ minimizes the worst case error among 
all algorithms that use $N=(L_1,L_2,\dots,L_n)$.  
The results of Bakhvalov and Smolyak can be found 
in~\cite{No88,NW08,NW10,TWW88}. However, 
$S(f)-A_n(f)$ is still a discontinuous linear   
functional and therefore its worst case error is infinite. Since this 
holds for all continuous $N$, we have   
$\enwall(S)=\infty$. Hence,  
$$ 
\widetilde e_n^{\rm \, all-wor}(S)=0<e_n^{\rm 
\,  all-wor}(S)=\infty 
\ \ \ \mbox{for all}\ \ \ n\in\N. 
$$ 
Although Theorem~\ref{adaptive} deals with adaptive 
information, it is known that adaptive information does \emph{not} 
help for many problems. This holds for linear operators 
$S$ defined over Hilbert spaces $F$ or if $S$ is a linear functional,  
whereas for linear operators defined 
over arbitrary normed spaces adaption may help at most by a factor of 
two. The reader may find a survey of such results 
in Chapter 4 of \cite{NW08}.  
 
\begin{rem}  \label{rem2.4} 
Seeing the proof of Theorem~\ref{adaptive}, we may think that  
the use of discontinuous linear functionals is useless for 
approximating continuous $S$. More precisely, assume that  
we use nonadaptive $N=(L_1,L_2,\dots,L_n)$ for which \emph{all}  
$L_j$'s are discontinuous linear functionals. What is the radius  
of $N$? Is it the same as the radius of zero information?  
This is not true. We now show that we can achieve the radius  
of nonadaptive information consisting of $n-1$ \emph{continuous} linear  
functionals.  Indeed, let $N^*=(L^*_2,L_3^*,\dots,L^*_{n})$  
 be nonadaptive continuous information, $L_j^*\in\lall$.  
Take $N=(L_1,L_1+L^*_2,\dots,L_1+L^*_n)$ with a discontinuous linear 
functional $L_1$. Then all $L_1+L_j^*$'s are discontinuous. 
However,  if we compute  
$y_1=L_1(f)$ and $y_j=L_1(f)+L^*_j(f)$  
then we also know $L^*_j(f)=y_j-y_1$ for $j=2,3,\dots,n$. 
Hence, we know $N^*(f)$ and therefore $r(N)\le r(N^*)$, as claimed.  
 
It is interesting to see what happens if we apply the proof of 
Theorem~\ref{adaptive} to $N$. Since $L_1$ is discontinuous we obtain 
$L_1^*=0$. However, $L_2+L^*_1$ on $B_1=\{f\in F\ |\ \ L_1(f)=y_1\}$ 
is $y_1+L^*_1$ and therefore it is \emph{continuous}. 
Then we replace $L_2$ by $L_2^*$. 
Similarly, all $L_j$ with $j\ge2$ will be replaced by 
$L_j^*$. The proof of Theorem~\ref{adaptive} shows that $r(N)=r(N^*)$. 
\end{rem} 
 
\begin{rem}\label{fevalworst} 
Assume now that $F$ is a space of functions $f:D\to \reals$. 
We consider the class $\lstd$ of all function evaluations given by
the linear functionals $L_x(f)=f(x)$ for $f\in F$.  
Let $e^{\rm std-wor}_n(S)$ denote the $n$th 
minimal worst case errors of algorithm $A_n$ with $L_j\in \lstd$, 
i.e., algorithms that use at most $n$ function evaluations. 

Assume that $L_x$ is \emph{discontinuous} on  
$$
F\, \cap \,\ker\, L_{x_1}\, \cap \, \ker\, L_{x_2}\, \cap\, 
 \dots \,\cap\,  \ker\,  L_{x_n}
$$
for all $x\in D\setminus\{ x_1, x_2, \dots , x_n \}$. 
For instance $F=L_2(D)\cap C(D)$ 
equipped with the $L_2(D)$ norm is such an example. 

Then it is easy to check that the proof of Theorem~\ref{adaptive} 
yields $L_j^*=0$ for all $j$. That is, 
$$ 
e^{\rm std-wor}_n(S)= 
e^{\rm std-wor}_0(S)= 
e^{\rm all-wor}_0(S).      
$$ 
Hence, in this case the use of function evaluations is completely useless. 

However, similarly as in Remark~\ref{rem2.4}, 
one can construct examples 
where all function evaluations are \emph{discontinuous} on $F$ 
but may be \emph{continuous} on 
$F\, \cap \,\ker L_{x_1}\, \cap \, \ker L_{x_2}\, \cap\, 
 \dots \,\cap\,  \ker  L_{x_n}$, and still they are useful. 
Indeed, consider 
$$
F= C([0,1])\, \cap\, \bigg\{ f \ \bigg|\ \ 
\int_0^1 f(x) \, {\rm d} x = \frac{f(0)+f(1)}2\, 
\bigg\} 
$$
equipped with the $L_2$ norm. 
Then the integration problem 
$S(f) = \int_0^1 f(x) \, {\rm d} x$ is continuous and all function evaluations
are discontinuous on $F$. Nevertheless, $L_1(f)=f(1)=S(f)/2$ on
$F\,\cap\,\ker \,L_0$ is continuous, and we can compute $S(f)$ exactly using 
two function values of~$f$.  
\end{rem} 

\begin{rem}   \label{rem5} 
For a continuous linear $S$ and a Banach space $F$,
Theorem~\ref{adaptive} can be proved modulo a factor of $\tfrac12$
by using the known relations between the Gelfand numbers $c_n(S)$ 
and the minimal errors $e_n^{\rm all-wor}(S)$,
see~\cite{NW08} for a survey of related results. 
In particular, we use that
$$ 
 c_n(S) \le \enwall(S) \le 2 c_n(S),
$$
for any linear and continuous operator $S$, see 
\cite[Section 5.4 of Chapter 4]{TWW88}.    

It is known, see \cite[Prop. 2.7.5]{CS90}, 
that the Gelfand numbers $c_n(S)$ are 
local in the sense that 
$$ c_n(S) = \sup_{M} c_n(S|_M)$$
where the supremum is taken over all finite dimensional subspaces 
$M$ and $S|_M$ is the restriction of $S$ to $M$.

Altogether, we obtain the following inequality:  
\begin{eqnarray*}
\enwall(S) 
  &\le& 
  2 c_n(S) = 2 \sup_{M} c_n(S|_M) \le 2 \sup_{M} \enwall(S|_M) 
  = 2 \sup_{M} \enwALL(S|_M) \\
  &\le& 2  \enwALL(S),
\end{eqnarray*}
i.e.
$$
\enwall(S)  \le 2 \cdot \enwALL(S) . 
$$

We return to general continuous not necessarily linear operators $S$. 
We conjecture that the worst case 
error $\enwall(S)$ is itself a local quantity
at least for compact operators~$S$, i.e., 
\begin{equation}\label{locerror}
\enwall(S)=\sup_{M}\,\enwall(S|_M),
\end{equation}
where again the supremum is 
taken over all finite dimensional 
subspaces~$M$ of ~$F$, 
and $S|_M$ is the restriction of $S$ to~$M$.
This would give 
another and a much shorter proof of Theorem~\ref{adaptive}.
Indeed, \eqref{locerror} implies 
$$
\enwall(S)=\sup_{M}\, \enwall(S|_M)=
  \sup_{M}\, \enwALL(S|_M) \le \enwALL(S),
$$
as claimed.

Although we do not know if~\eqref{locerror} holds,
we easily conclude from 
the local property of the Gelfand numbers that
at least the weak local property holds for a continuous linear $S$
and a Banach space $F$, namely
$$ 
\enwall(S) \le c_n(S)= \sup_{M}\,c_N(S|_M)\le 2\,
\sup_M\,\enwall(S|_M). 
$$ 
It is also known that the approximation 
numbers $a_n(S)$ are local as long as $S$ is
a compact operator or~$G$ is a dual space, 
see \cite[Prop. 2.7.1 and 2.7.3]{CS90}.
That translates into the fact that for
linear nonadaptive algorithms in the worst case setting, 
the use of discontinuous functionals  
does not help. A weak form (with a factor 5) is true 
for arbitrary $G$ and $S$, see \cite[Prop. 2.7.4]{CS90}.
\end{rem} 

\begin{rem} 
Heinrich~\cite{He89} proves relations between linear $n$-widths 
and approximation numbers and shows that they coincide for 
compact and absolutely convex subsets of a normed space. 
There is also an example showing that, in general, 
for relatively compact 
absolutely convex sets equality does not hold.
The spirit of these results is similar, but there is a difference 
as can be seen from Proposition 1.3 of that paper: 
The aim is to compare general or continuous linear information 
applied to $g=S(f) \in G$, while we compare general or 
continuous linear information applied to $f \in F$. 
\end{rem}  
 
\section{Randomized Setting}   \label{s3}  
We now deal with randomized algorithms.  
We consider,  
as in \cite[Theorem 4.42]{NW08}, only measurable algorithms.  
Hence we use the following definitions.  
 
A randomized algorithm $A$ is a pair consisting of a probability space 
$(\Omega, \Sigma, \mu)$ and  
a family $(N_\omega, \phi_\omega )_{\omega \in \Omega}$ 
of mappings such that the following holds: 
 
\begin{enumerate} 
\item 
For each fixed $\omega\in\Omega$,   
the mapping $A_\omega=\phi_\omega \circ N_\omega$ 
is a deterministic algorithm defined as before, 
based on adaptive information $N_\omega$ consisting 
of linear functionals from a class $\Lambda$. 
 
\item 
Let $n(f,\omega)$ be the cardinality of the  
information $N_\omega$ for $f \in F$. We assume that 
the function               $n$ is measurable. 
\item 
The mapping $(f, \omega)\mapsto \phi_\omega (N_\omega(f)) \in G$  
is measurable.
\end{enumerate} 
 
Let $A$ be a randomized algorithm. Then 
the cardinality of $A$ is defined as 
$$ 
n(A) = \sup_{\Vert f \Vert_F < 1} \int_\Omega 
n(f, \omega) \, {\rm d}\mu (\omega), 
$$ 
whereas the error of $A$ in the randomized setting is 
$$ 
e^{\rm ran} (A) = \sup_{\|f\|_F < 1 } \left( \int_\Omega 
\Vert S(f) - \phi_\omega (N_\omega (f)) 
\Vert^2 \, {\rm d}\mu (\omega) \right)^{1/2}. 
$$ 
 
For $n \in \N$,  we define the $n$th minimal error of $S$ in the 
randomized setting by  
$$ 
e^{\rm all-ran}_n (S) = \inf \{ e^{\rm ran} (A) \, : \, n(A) \le n \} , 
$$ 
if $A$ uses linear functionals from $\Lambda = \lall$. Similarly we define  
$ \wt e^{\rm \ all-ran}_n (S)$ with 
$\Lambda = \tlall$. 
 
Obviously, we can interpret deterministic algorithms as randomized 
with a singleton $\Omega$. That is why the $n$th minimal errors in the 
randomized setting cannot be larger than the $n$th minima errors in  
the worst case setting, 
$$ 
\wt e^{\rm \ all-ran}_n(S)\le \wt e^{\rm \ all-wor}_n(S) 
\ \ \ \mbox{and}\ \ \  
e^{\rm all-ran}_n(S)\le e^{\rm all-wor}_n(S). 
$$   
 
Basically, there exists only one proof technique  
to obtain lower bounds for randomized algorithms,  
and this technique goes back to Bakhvalov, see Section 4.3.3  
of \cite{NW08}.  
The main point is to observe that the errors in the randomized  
setting cannot be smaller than the errors in the average case setting 
for an arbitrary probability measure on $F$.    
 
If we restrict ourselves  to measurable randomized algorithms,  
then  
\begin{equation}  \label{eq03}    
e^{\rm all-ran}_n (S)  \ge \tfrac{\sqrt{2}}{2}\,  
e^{\rm avg} (2n,\rho ) ,  
\end{equation} 
where $e^{\rm avg}(2n,\rho)$ denotes the $2n$-th minimal average case 
error of deterministic algorithms that use at most $2n$ linear 
functionals from $\lall$ and $\rho$ is  
an arbitrary (Borel) probability measure on $F$,  
for more details see Lemma 4.37 and Remark 4.41 of \cite{NW08}. 
The next result is identical with  
Theorem 4.42 from \cite{NW08}, see also \cite{No92}.  
It is stated here with a proof since we need  
a small modification of this result.  
 
\begin{lem}      \label{lemma7}   
Assume that $S: F \to G$ is a compact linear operator  
between Hilbert spaces 
$F$ and $G$. Then 
$$ 
\tfrac{1}{2} \,  e^{\rm all-wor}_{4n-1}  (S)  \le 
e^{\rm all-ran}_n (S) . 
$$ 
\end{lem}  
 
\begin{proof} 
We know that $S(e_i) = \sigma_i \tilde e_i$ with orthonormal 
$\{ e_i \}$ in $F$ and $\{ \tilde e_i \}$ in $G$. Here, $\{\sigma_i\}$  
is a sequence of non-increasing singular values $\sigma_i$ of $S$ and 
$\lim_i\sigma_i=0$.  
We also know that $e^{\rm \ all-wor}_n (S) = \sigma_{n+1}$. 
For $m>n$, consider the normed 
$(m-1)$-dimensional Lebesgue measure~$\rho_m$ on the unit sphere 
$E_m = \{ \sum_{i=1}^m \alpha_i e_i \, : \, \alpha_i \in  
\reals, \ \sum_{i=1}^m \a_i^2 
=1  \}$. 
Then 
$$ 
A_n^* \biggl( \sum_{i=1}^\infty \alpha_i e_i \biggr) = 
\sum_{i=1}^n \sigma_i \alpha_i \tilde e_i 
$$ 
is the optimal algorithm using continuous  
linear information of cardinality $n$. 
This is true for the worst case setting, with error $\sigma_{n+1}$, as well as 
for the average case setting with respect to~$\rho_m$. 
Hence 
$$ 
e^{\rm avg} (n, \rho_m)^2 = 
\int_{E_m}  \sum_{i=n+1}^m \sigma_i^2 \alpha_i^2 \, d \rho_m (\alpha) . 
$$ 
Since $\int_{E_m}  \alpha_i^2 \, d\rho_m (\alpha) = 1/m$ 
we obtain 
$$ 
e^{\rm avg} (n, \rho_m ) ^2 = 
\frac{1}{m} \sum_{i=n+1}^m \sigma_i^2 . 
$$ 
If we put $m=2n$ then we obtain 
$$ 
e^{\rm avg} (n, \rho_{2n} )  \ge  \tfrac12\,\sqrt{2}\,\sigma_{2n} . 
$$ 
Together with \eqref{eq03}, we obtain 
$$ 
e^{\rm \ all-ran}_n ( S)  \ge \tfrac12\,\sqrt{2}\, e^{\rm avg} (2n, \rho_{4n}) 
\ge \tfrac{1}{2} \,  \sigma_{4n}    
= \tfrac{1}{2} \, e^{\rm all-wor}_{4n-1} (S) .  
$$ 
\end{proof} 
 
We stress that in the proof we only use finite dimensional subspaces 
of $F$  
and linear functionals on such subspaces. Of course,  
for finite dimensional spaces,  
we have $\tlall = \lall$ and therefore we obtain  
the following result.  
 
\begin{thm}  \label{theorem8}  
Assume that $S: F \to G$ is a compact operator between  
Hilbert spaces. Then  
$$ 
\tfrac{1}{2} \,  e^{\rm \,  all-wor}_{4n-1}(S)\le 
\wt e^{\rm \,  all-ran}_n (S) \le 
e^{\rm \,  all-wor}_n (S). 
$$ 
\end{thm}  
 
\noindent  
In this sense, randomization as well as allowing  
discontinuous linear functionals from  
$\tlall$ does not (essentially) help for the 
approximation of 
compact operators between Hilbert spaces.  

\begin{rem}  
Assume that $S:F\to G$ is linear and $F$ and $G$ are normed spaces.  
We do not know whether  
$$ 
\lim_{n \to \infty} e^{\rm \ all-ran}_n (S) = 0 
$$ 
implies that $S$ is compact.  
It is shown in \cite{J71}  
that the embedding $I: \ell_1 \to \ell_\infty$ is a universal 
non-compact operator in the sense that
it factors through any non-compact linear bounded operator 
$S: F \to G$ between Banach spaces $F$ and $G$:
$$
\begin{CD}
F @>S>> G\\
@AVAA @VVUV\\
\ell_1 @>I>> \ell_\infty\ .
\end{CD}
$$ 
Here, $I=USV$ for some 
linear bounded operators $U$ and $V$.
It follows that 
$$
 e^{\rm \ all-ran}_n (I) = e^{\rm \ all-ran}_n (USV) \le \|U\| \, e^{\rm \ all-ran}_n (S) \, \|V\|.
$$
Thus it is sufficient to decide whether 
$$ 
\lim_{n \to \infty} e^{\rm \ all-ran}_n (I) > 0 \ \ \ \mbox{or}\ \ \ 
\lim_{n \to \infty} e^{\rm \ all-ran}_n (I) = 0.
$$ 
\end{rem}  

\begin{rem}
 Also in the randomized setting it would be very interesting to know
 whether the error $\enrall(S)$ is a local quantity, i.e. whether
 $$
   \enrall(S) = \sup_{M}\, \enrall(S|_M),
 $$ 
 or at least
 $$
   \enrall(S) \le c\, \sup_{M}\, \enrall(S|_M)
 $$ 
 for some constant $c$ independent of $n$.
 This would lead to an analogue of 
Theorem~\ref{adaptive} in the randomized setting.
\end{rem}

\begin{rem}  
It is interesting to mention that there is a continuous  
\emph{nonlinear} operator  
$S: F \to G$ for normed spaces $F$ and $G$
which is solvable in the randomized setting  
but not in the worst case setting, i.e., 
$$ 
\lim_n\wt e^{\rm \ all-wor}_n(S)=\lim_ne^{\rm \ all-wor}_n(S)>0\ \  
\ \ \ 
\mbox{and}\ \ \ 
\lim_ne^{\rm all-ran}_n(S)=0. 
$$ 
Obviously, the first equality above follows from  
Theorem~\ref{adaptive}. Indeed, let  
$$ 
S: F:=\ell_1 \to G:=\reals, \qquad S(x) = \Vert x \Vert_2^2\ \ \ 
\mbox{for all}\ \ \ x\in F. 
$$ 
Clearly, $S$ is a continuous nonlinear functional.  
Note that the constant algorithm $\tfrac12$ has the worst case  
error~$\tfrac12$ since we have 
$S(x)\in[0,1]$ for $x$ from the unit ball of $\ell_1$.
Let 
$$ 
c:= \lim_{n\to\infty}e^{\rm \ all-wor}_n(S). 
$$ 
Then $c\le \tfrac12$.    
We now show that $c>0$. 
We will use the known result of Kashin about the Gelfand 
width  
$c_n(B_1^m,\ell_2^m)$
for the unit ball $B_1^{\,m}$ of $x\in 
\reals^m$ with $\|x\|_1\le 1$, and with the error measured 
in the $\ell_2$ norm. Namely, 
there are two positive numbers $c_{1,2}$ and $C_{1,2}$ such that  
for all $n<m$ we have 
$$ 
c_{1,2}\,\min\left(1,\frac{\ln(m/n)+1}{n}\right)^{1/2}\le 
   c_n(B_1^{\,m},\ell_2^m) \le
C_{1,2}\,\min\left(1,\frac{\ln(m/n)+1}{n}\right)^{1/2}. 
$$ 
This means that $\lim_{m\to\infty}c_n(B_1^{\,m},\ell_2)\ge
c_{1,2}$. Using the definition of the Gelfand width, we conclude that  
$$ 
\lim_{n\to \infty}\  
\inf_{L_1,L_2,\dots,L_n\in\tlall}\ \  
\sup_{x\in\ell_1, \ L_j(x)=0, \,j=1,2,\dots,n,\ \|x\|_1\le1}\|x\|_2
\ge c_{1,2}. 
$$ 
 
Take now arbitrary adaptive $N=(L_1,L_2,\dots,L_n)$ with linear  
functionals~$L_i$. Then there exists $x$ 
in the unit ball of $\ell_1$ such that 
$L_j(x)=0$ for $j=1,2,\dots,n$ and $\|x\|_2\ge c_{1,2}/2$. 
Let $\alpha\in[0,1]$. Then  
$\alpha\,x$ belongs to the unit ball of $\ell_1$ and $N(\alpha\,x)=0$. 
For an arbitrary algorithm $A=\phi_n(N(\cdot))$ we have 
$A_n(\alpha\,x)=\phi_n(0)$ and 
$S(\alpha\,x)-\phi_n(0)=\alpha\|x\|_2^2-\phi_n(0)$. 
Therefore 
$$ 
e(A_n)\ge\max_{\alpha\in[0,1]}|\alpha\|x\|^2_2-\phi_n(0)| 
\ge\tfrac12\|x\|^2_2\ge\tfrac14c_{1,2}. 
$$ 
Hence $e^{\rm all-wor}_n(S)\ge\tfrac14c_{1,2}$ for all $n$. 
This proves that $c>0$, as claimed.  

In the randomized setting, consider the (random) linear functional  
$$ 
L(x) = \sum_{k=1}^\infty (\pm)\, x_k 
$$ 
with random and independent signs of probability $\tfrac12$.
The variance of the random variable $L(x)$ is  
$\sigma^2 (L(x)) = \Vert x \Vert_2^2$ and it 
can be easily estimated with independent copies $L_i$ of $L$.  
It is well known that the ``empirical variance'' 
$$ 
A_n(x)= \frac{1}{n-1} \sum_{i=1}^n \left( L_i -
\sum_{k=1}^n L_k \right)^2  
$$ 
with independent copies $L_i$ of $L$ has expectation
$S(x)=\Vert x \Vert_2^2$ and variance
$$ 
\sigma^2 (A_n(x)) \le \frac{1}{n} \cdot \Vert x \Vert_4^4
\le \frac1n\cdot \|x\|_1^4.
$$
Therefore
$$
e^{\rm ran}(A_n)=\sup_{\|x\|_{\ell_1} < 1 }\sigma(A_n(x))\le
\frac1{\sqrt{n}},
$$ 
and 
$\lim_ne^{\rm all-ran}_n(S)=0$, as claimed. 
\end{rem}  
 
\section{Function Values}   \label{s4}  
 
So far we did not assume that $F$ is a space of functions  
and we only compared continuous linear information with  
arbitrary linear information. Now we assume that $F$ is a normed   
space of functions $f:D\to\reals$ for some nonempty set $D$. 
  
Let $L_x(f)=f(x)$ for all $f\in F$ and $x\in D$. Since $f(x)$ is well 
defined for all $f\in F$ and $x\in D$, the functionals $L_x$'s are 
linear but \emph{not} necessarily continuous. 
This class of information is called \emph{standard} and denoted by  
$\lstd$. Obviously $\lstd\subseteq \tlall$, however, $\lstd\subseteq 
\lall$ only if all $L_x$'s are continuous.
Still Theorems~\ref{adaptive} and~\ref{theorem8} apply in this case.  

We now consider a more general case when $F$ is a space of
equivalence classes of functions $f:D\to \reals$.    
A major example is $F=L_2 (D)$. Then $L_x$ is not 
even well defined for $f \in F$. On the other hand, we know that for 
$F=L_2(D)$ the functional $L_x(\tilde f)=\tilde f(x)$ is well 
defined for each $\tilde f \in f$ and all $x$ from $D$. 
Here $\tilde f\in f$ means that the well defined function
$\tilde f$ is in the equivalence class $f$.
This type of information is 
successfully used for multivariate integration by 
the standard Monte Carlo algorithm $M_n$ in the randomized setting.  
Here for $D=[0,1]^d$, we have $M_n(f)=n^{-1}\sum_{j=1}^nf(x_j)$ 
for independent and uniformly distributed points $x_j$.
Then 
$$
M_n (\tilde f_1 ) = M_n (\tilde f_2 )
 \qquad a.s. 
$$
if $\tilde f_1, \tilde f_2 \in f$ and one usually 
uses only algorithms with this property. 
We would like to extend the analysis presented in the 
previous section also for the case when
the elements of $F$ are equivalence classes of functions.

We argue as follows.  
As in Section~\ref{s3}, we assume that  
$S: F \to G$ is a compact linear operator between 
the Hilbert spaces $F$ and $G$.  
We know that then  
$$ 
S(e_i) = \sigma_i  e_i'
$$  
with a non-increasing sequence of singular values $\sigma_i$ of $S$, 
$\lim_i\sigma_i=0$, with orthonormal 
$\{ e_i \}$ in $F$ and orthonormal $\{ e_i' \}$ in $G$.
We also know that $e^{\rm all-wor}_n (S) = \sigma_{n+1}$. 
Then 
$$ 
A_n^* \biggl( \sum_{i=1}^\infty \alpha_i e_i \biggr) = 
\sum_{i=1}^n \sigma_i \alpha_i e_i'
$$ 
is an optimal algorithm using continuous linear information 
of cardinality $n$. 
This is also true if we replace $F$ by the  
$n+k$ dimensional space $V_{n+k}={\rm span}(e_1, e_2, \dots 
e_{n+k})$. Hence 
$$ 
\wt e^{\rm \ all-wor}_n (V_{n+k}) = e^{\rm all-wor}_n (V_{n+k}) = \sigma_{n+1} 
=e^{\rm all-wor}_n (S)\ \ \ \mbox{for all}\ \ \ k\ge1. 
$$ 
Suppose that the functions $\tilde e_i$ are elements 
of the equivalence classes $e_i$,
for all $i=1,2, \dots , n+k$. 
Then we have functions in 
$ \widetilde V_{n+k}={\rm span}(\tilde e_1, \tilde e_2, \dots 
\tilde e_{n+k})$ 
that  
are well defined everywhere. With this assumption we only make the oracle  
more powerful, i.e., the lower bound is even stronger.  
In this sense we can think of~$\lstd$ as a subset  
of~$\tlall$. 
By $e_n^{\rm \ std-wor}(S)$ we denote the $n$ minimal worst case 
errors of algorithms that use at most $n$ function values for 
approximating $S$ over $\widetilde V_{n+k}$. 
We obtain the following corollary.  
 
\begin{cor}    \label{cor10}  
Assume that $S:F\to G$ is a compact linear operator between  
Hilbert spaces $F$ and $G$. Then  
$$ 
e_n^{\rm \ std-wor} (S) \ge  \enwall(S)\ \ \  
\mbox{for all} \ \ \  n\in\nat .  
$$  
\end{cor}  

The same can be said for randomized algorithms. 
In this case we take $k\ge 3n$, update the definition of 
$e_n^{\rm \ std-ran}(S)$ and obtain the following corollary. 
 
\begin{cor}   \label{cor11}  
Assume that $S: F \to G$ is a compact linear operator between  
Hilbert spaces $F$ and $G$. Then  
$$ 
e^{\rm std-ran}_n ( S)  
\ge \tfrac{1}{2} \,  e^{\rm all-wor}_{4n-1}  (S) . 
$$ 
\end{cor}  

\begin{rem}  
Note that we did \emph{not} compare  
$e^{\rm std-ran}_n (S)$ with $e^{\rm std-wor}_n (S)$.  
We stress that randomization may help  
a lot for the class $\lstd$.  
This holds if function evaluations are continuous and also  
if function evaluations are not continuous.  
Examples for both cases can be found in Chapter 17 of  
\cite{NW10}.  

A major example is the embedding of a function space into another (larger) 
function space. The literature is very rich, see, e.g., 
\cite{He92,He06,He09a,He09b,He11,Ma91,Ma93,NT06,NW08,Tr10,Vy07}.
One can study the classes $\lall$, $\tlall$ as well as 
$\lstd$ in the worst case setting and in the randomized setting. 
In the randomized setting we do \emph{not} know whether 
$\lall$ and $\tlall$ always lead to the same results since in 
Theorem~\ref{theorem8} we assume that both $F$ and $G$ are Hilbert spaces. 
It is open what happens for general normed spaces $F$ and $G$
and if an analogue of Theorem~\ref{adaptive} for the worst  
case setting also holds in the randomized setting.

We also add that the 
lower bounds in the randomized setting of Heinrich and Math\'e 
for specific spaces $F$ and $G$ 
are valid 
not only for $\lall$ but also for $\tlall$, see \cite{He09a,He09b,He11,Ma91}.
The reason is similar as above in the proof of 
Theorem~\ref{theorem8}. Namely, finite dimensional subspaces 
of $F$ can be used 
for the lower bounds and here all linear functionals are continuous. 
\end{rem} 

\medskip 

{\bf Acknowledgement: } 
We thank Stefan Heinrich for valuable remarks.

\end{document}